\documentclass[english]{article}
\usepackage[T1]{fontenc}
\usepackage[latin9]{inputenc}
\usepackage{color}
\usepackage{refstyle}
\usepackage{mathrsfs}
\usepackage{enumitem}
\usepackage{amsmath}
\usepackage{amsthm}
\usepackage{amssymb}

\makeatletter

\AtBeginDocument{\providecommand\secref[1]{\ref{sec:#1}}}
\AtBeginDocument{\providecommand\eqref[1]{\ref{eq:#1}}}
\AtBeginDocument{\providecommand\lemref[1]{\ref{lem:#1}}}
\AtBeginDocument{\providecommand\propref[1]{\ref{prop:#1}}}
\AtBeginDocument{\providecommand\corref[1]{\ref{cor:#1}}}
\AtBeginDocument{\providecommand\remref[1]{\ref{rem:#1}}}
\AtBeginDocument{\providecommand\thmref[1]{\ref{thm:#1}}}
\AtBeginDocument{\providecommand\exaref[1]{\ref{exa:#1}}}
\providecommand{\tabularnewline}{\\}
\RS@ifundefined{subsecref}
  {\newref{subsec}{name = \RSsectxt}}
  {}
\RS@ifundefined{thmref}
  {\def\RSthmtxt{theorem~}\newref{thm}{name = \RSthmtxt}}
  {}
\RS@ifundefined{lemref}
  {\def\RSlemtxt{lemma~}\newref{lem}{name = \RSlemtxt}}
  {}

\theoremstyle{plain}
\newtheorem{thm}{\protect\theoremname}[section]
  \theoremstyle{plain}
  \newtheorem{prop}[thm]{\protect\propositionname}
  \theoremstyle{definition}
  \newtheorem{defn}[thm]{\protect\definitionname}
  \theoremstyle{plain}
  \newtheorem{lem}[thm]{\protect\lemmaname}
  \theoremstyle{plain}
  \newtheorem{cor}[thm]{\protect\corollaryname}
  \theoremstyle{definition}
  \newtheorem{example}[thm]{\protect\examplename}
  \theoremstyle{remark}
  \newtheorem{rem}[thm]{\protect\remarkname}

\@ifundefined{date}{}{\date{}}
\AtBeginDocument{}\AtBeginDocument{}\AtBeginDocument{\providecommand\remref[1]{\ref{rem:#1}}}\AtBeginDocument{}

\AtBeginDocument{\providecommand\propref[1]{\ref{prop:#1}}}
\RS@ifundefined{propref}{\newref{prop}{name = Proposition~,names = propositions~}}{}

\AtBeginDocument{\providecommand\corref[1]{\ref{cor:#1}}}
\RS@ifundefined{corref}{\newref{cor}{name = Corollary~,names = Corollaries~}}{}

\AtBeginDocument{\providecommand\corref[1]{\ref{rem:#1}}}
\RS@ifundefined{remref}{\newref{rem}{name = Remark~,names = Remarks~}}{}

\AtBeginDocument{\renewcommand\secref[1]{Section~\ref{sec:#1}}}

\AtBeginDocument{\providecommand\exaref[1]{\ref{exa:#1}}}
\RS@ifundefined{exaref}{\newref{exa}{name = Example~,names = Examples~}}{}

\usepackage{tikz}
\usepackage{ytableau}
\usepackage{tocbibind}
\usepackage{youngtab}
\usepackage{authblk}
\usetikzlibrary{matrix}

\DeclareMathOperator{\Hom}{Hom}

\DeclareMathOperator{\id}{id}

\DeclareMathOperator{\tr}{tr}
\DeclareMathOperator{\Rad}{Rad}

\DeclareMathOperator{\PT}{\mathcal{PT}}
\DeclareMathOperator{\E}{\mathcal{E}}

\DeclareMathOperator{\SE}{\mathcal{SE}}

\DeclareMathOperator{\T}{\mathcal{T}}

\DeclareMathOperator{\Stab}{Stab}
\DeclareMathOperator{\Ind}{Ind}
\DeclareMathOperator{\Res}{Res}

\DeclareMathOperator{\op}{op}

\DeclareMathOperator{\sgn}{sgn}
\DeclareMathOperator{\pd}{pd}
\DeclareMathOperator{\Ext}{Ext}

\DeclareMathOperator{\gd}{glDim}

\DeclareMathOperator{\VS}{\bf{VS}}
\DeclareMathOperator{\ds}{ds}

\def\RSlemtxt{Lemma~}
\def\RSthmtxt{Theorem~}

\usepackage{babel}
\providecommand{\corollaryname}{Corollary}
  \providecommand{\definitionname}{Definition}
  \providecommand{\examplename}{Example}
  \providecommand{\lemmaname}{Lemma}
  \providecommand{\propositionname}{Proposition}
  \providecommand{\remarkname}{Remark}
 
\providecommand{\theoremname}{Theorem}
\providecommand{\examplename}{Example}

\makeatother

\usepackage{babel}
  \providecommand{\corollaryname}{Corollary}
  \providecommand{\definitionname}{Definition}
  \providecommand{\examplename}{Example}
  \providecommand{\lemmaname}{Lemma}
  \providecommand{\propositionname}{Proposition}
  \providecommand{\remarkname}{Remark}
\providecommand{\theoremname}{Theorem}

\begin{document}

\title{The global dimension of the algebra of the monoid of all partial
functions on an $n$-set as the algebra of the EI-category of epimorphisms
between subsets}

\author{Itamar Stein\thanks{This research was carried out while the author was a PHD student of
Prof. Stuart Margolis at Bar Ilan University. The author's research
was supported by Grant No. 2012080 from the United States-Israel Binational
Science Foundation (BSF). }\\
 Mathematics Unit\\
Shamoon College of Engineering\\
Israel\\
Steinita@gmail.com}
\maketitle
\begin{abstract}
We prove that the global dimension of the complex algebra of the monoid
of all partial functions on an $n$-set is $n-1$ for all $n\geq1$.
This is also the global dimension of the complex algebra of the category
of all epimorphisms between subsets of an $n$-set. In our proof we
use standard homological methods as well as combinatorial techniques
associated to the representation theory of the symmetric group. As
part of the proof, we obtain a partial description of the Cartan matrix
of these algebras.
\end{abstract}

\section{Introduction}

Let $\mathcal{D}$ be a finite monoid, or more generally, a finite
category. It is of interest to study the complex category algebra
$\mathbb{C}\mathcal{D}$ and its representations. Central objects
of research interest are the Jacobson radical, ordinary quiver, quiver
presentation, Cartan matrix, global dimension etc. Note that all these
invariants are virtually trivial in the semisimple case so these questions
does not arise in ordinary group representation theory. However, unlike
(complex) group algebras, category or even monoid algebras are seldom
semisimple. Monoids with natural combinatorial structure are clearly
of major interest. In this paper we study the monoid algebra $\mathbb{C}\PT_{n}$
where $\PT_{n}$ is the monoid of all partial functions on an $n$
element set. Note that in this paper composition of functions is done
from right to left. $\PT_{n}$ is fundamental in monoid theory, for
instance, a major part of \cite{Ganyushkin2009b} is devoted to its
study. Denote by $\mathscr{J}$ the usual Green's relation ($a\mathscr{J}b$
if they generate the same principal ideal, see \cite[Chapter 2]{Howie1995}).
In \cite{Putcha1996} Putcha essentially observed that $\mathbb{C}\PT_{n}$
is co-directed, that means that all the arrows in the quiver are goings
downwards (with respect to the natural partial order on irreducible
representations induced from the $\mathscr{J}$ order). In \cite{Stein2016}
the author proved that $\mathbb{C}\PT_{n}$ is isomorphic to the complex
algebra of $\E_{n}$, the category of all epimorphisms between subsets
of an $n$ element set. Studying representations of $\E_{n}$ is apparently
easier then representations of $\PT_{n}$ because the underlying graph
structure gives us additional information. Using this isomorphism
we were able to give an explicit description of the quiver of $\mathbb{C}\PT_{n}$
and $\mathbb{C}\E_{n}$ as well as some other observations. In this
paper we continue to study the representation theory of these algebras.
The main goal of this paper is finding the global dimension of $\mathbb{C}\PT_{n}\simeq\mathbb{C}\E_{n}$.
We denote the global dimension of an algebra $A$ by $\gd A$. It
is the supremum over the minimal lengths of all projective resolutions
of modules over the algebra. We remark that Steinberg \cite{Steinberg2016}
proved that the global dimension of $\mathbb{C}\T_{n}$ (where $\T_{n}$
is the monoid of all total functions on an $n$ element set) is $n-1$.
Let $M$ be a (finite) regular monoid. A theorem of Nico \cite{Nico1972}
says that the global dimension of $\mathbb{C}M$ is bounded above
by $2k$ where $k$ is the maximal length of a chain in the $\mathscr{J}$
order. For algebras with directed or co-directed quivers, the bound
is at most $k$. Using this result and the observations of Putcha
one can prove that $\gd\mathbb{C}\PT_{n}\leq n-1$. Since the global
dimension is bounded above by the maximal path in the quiver, this
upper bound also follows from the explicit description of the quiver.
In this paper we prove that $\gd\mathbb{C}\E_{n}=\gd\mathbb{C}\PT_{n}=n-1$
for $n\geq1$. For this we use another fundamental invariant of an
algebra, the Cartan matrix. Let $A$ be a finite dimensional $\mathbb{C}$-algebra
with $r$ irreducible representations (up to isomorphism) denoted
$S(1),\ldots S(r)$. The Cartan matrix of $A$ is an $r\times r$
integer matrix whose $(i,j)$ entry is the number of times that $S(i)$
appears as a Jordan-Hölder factor in the projective cover of $S(j)$.
In \secref{CartanMatrixOfEICategoryAlgebra} we give a description
of the Cartan matrix of any EI-category algebra. A category $\mathcal{D}$
is called an EI-category if every endomorphism monoid of $\mathcal{D}$
is a group (so we can speak of the \emph{automorphism groups} of $\mathcal{D}$).
By description, we mean that we reduce the description of the Cartan
matrix to a question in the representation theory of the automorphism
groups. In \secref{RepresentationTheoryOfPTnAndEn} we give some background
and observations on $\mathbb{C}\E_{n}\simeq\mathbb{C}\PT_{n}$. $\E_{n}$
is an EI-category whose automorphism groups are $S_{k}$ for $0\leq k\leq n$
(where $S_{k}$ is the symmetric group on a $k$-element set). Moreover,
the irreducible representations of this algebra are in one-to-one
correspondence with Young diagrams with $k$ boxes where $0\leq k\leq n$.
Therefore, the Cartan matrix $C$ is a $p\times p$ matrix where 
\[
p=\sum_{k=0}^{n}p(k)
\]
and $p(k)$ is the number of integer partitions of $k$. With a natural
ordering of rows and columns, we observe that $C$ is a block upper
unitriangular matrix. Using results from \cite{Stein2016}, it is
easy to describe the first superdiagonal block of $C$ using standard
branching rules for Young diagrams. In \secref{TheSecondBlock} we
use the description of the Cartan matrix obtained in \secref{CartanMatrixOfEICategoryAlgebra}
in order to give a more concrete description of the second block superdiagonal
of the Cartan matrix, again, using branching rules. In \secref{TheGlobalDimension}
we use this description and standard homological methods such as the
long exact sequence theorem in order to prove that the projective
dimension of the simple module corresponding to the diagram $[2,1^{n-2}]$
is $n-1$. This proves that the global dimension is also $n-1$.

\textbf{Acknowledgments:} The author would like to thank Prof. Volodymyr
Mazorchuk for a very useful discussion on the representation theory
of $\mathbb{C}\PT_{n}$. In particular, the author is grateful for
his suggestion to investigate the projective dimension of the simple
module corresponding to the Young diagram $[2,1^{n-2}]$. The author
also thanks the referee for his\textbackslash{}her valuable comments
and remarks.

\section{Preliminaries}

\subsection{Representations of algebras}

Let $A$ be an algebra. We will only discuss unital, finite dimensional
$\mathbb{C}$-algebras. Likewise, when we say that $M$ is a module
over $A$ (or an $A$-module, or an $A$-representation) we mean that
$M$ is a finite dimensional left module over $A$. Details and proof
for facts in this subsection can be found in \cite{Assem2006}.

In this paper, we will mainly discuss category algebras. We will only
discuss finite categories. For every finite category $\mathcal{D}$
denote by $\mathcal{D}^{0}$ its set of objects, by $\mathcal{D}^{1}$
its set of morphisms and by $\mathcal{D}(a,b)$ the hom-set of all
morphisms with domain $a$ and range $b$. The \emph{category algebra}
$\mathbb{\mathbb{C}}\mathcal{D}$ is defined in the following way.
It is a vector space over $\mathbb{C}$ with basis the morphisms of
$\mathcal{D}$, that is, it consists of all formal linear combinations
\[
\{k_{1}m_{1}+\ldots+k_{n}m_{n}\mid k_{i}\in\mathbb{C},\,m_{i}\in\mathcal{D}^{1}\}.
\]
The multiplication in $\mathbb{C}\mathcal{D}$ is the linear extension
of the following:
\[
m^{\prime}\cdot m=\begin{cases}
m^{\prime}m & \exists m^{\prime}\cdot m\\
0 & \text{otherwise}.
\end{cases}
\]
Where $\exists m^{\prime}\cdot m$ mean that the composition of the
morphisms $m^{\prime}$ and $m$ is defined. Since a monoid is a category
with one object, this definition also gives a definition for monoid
algebras. In this case the monoid algebra contains linear combinations
of elements of the monoid with the obvious multiplication. It will
be often convenient to omit the field and call a $\mathbb{C}\mathcal{D}$-module
just a $\mathcal{D}$-module (or a $\mathcal{D}$-representation).

Given some $A$-module $M$, we denote by $\Hom_{A}(M,-)$ the usual
hom functor from the category of all finite dimensional $A$-modules
to the category of $\mathbb{C}$ vector spaces. Recall that an $A$-module
$P$ is called \emph{projective} if $\Hom_{A}(P,-)$ is an exact functor,
or equivalently, if $P$ is a direct summand of a free module $A^{k}$
for some $k\in\mathbb{N}$.\textcolor{black}{{} Recall that two idempotents
}$e,f\in A$\textcolor{black}{{} are called orthogonal if }$ef=fe=0$.\textcolor{black}{{}
A non-zero idempotent} $e\in A$ is called \emph{primitive} if it
is not a sum of two non zero orthogonal idempotents. This is equivalent
to $eAe$ being a local algebra (i.e., an algebra with no non-trivial
idempotents). A \emph{complete set of primitive orthogonal idempotents
}is a set of primitive, mutually orthogonal idempotents $\{e_{1},\ldots,e_{r}\}$
whose sum is $1$. It is well known that every indecomposable projective
module is isomorphic to $Ae$ for some primitive idempotent $e\in A$.
Moreover, every simple $A$-module $S$ is isomorphic to $Ae/\Rad Ae$
for some primitive idempotent $e\in A$ (where $\Rad M$ denotes the
radical of the module $M$). Therefore, we can associate with every
primitive idempotent an indecomposable projective module and a simple
module. Two primitive idempotents $e,f$ are called \emph{equivalent}
if the associated indecomposable projective modules are isomorphic,
i.e, $Ae\simeq Af$. This happens precisely when the associated simple
modules are isomorphic, i.e, $Ae/\Rad(Ae)\simeq Af/\Rad(Af)$.

We recall that $\Ext^{n}(M,-)$ is the $n$-th right derived functor
of $\Hom(M,-)$ where $n\in\mathbb{N}$. For a detailed explanation
on the $\Ext$ functor, see \cite[Chapters 6-7]{Rotman2009}. What
we will need to know about this functor are the following facts: $\Ext^{n}(-,-)$
is an additive functor in both arguments. If $P$ is a projective
$A$-module then $\Ext^{n}(P,N)=0$ for every $n\in\mathbb{N}$ and
every $A$-module $N$. If 
\begin{equation}
0\to N\to K\to M\to0\label{eq:ShortExactSeq}
\end{equation}
is a short exact sequence and $\Ext^{1}(M,N)=0$ then we must have
that $K\simeq M\oplus N$. Moreover, for every short exact sequence
as in \eqref{ShortExactSeq} and for every $A$-module $L$, we can
construct a long exact sequence
\begin{align*}
0 & \to\Hom(M,L)\to\Hom(K,L)\to\Hom(N,L)\to\\
 & \to\Ext^{1}(M,L)\to\Ext^{1}(K,L)\to\Ext^{1}(N,L)\to\\
 & \ldots\\
 & \to\Ext^{m}(M,L)\to\Ext^{m}(K,L)\to\Ext^{m}(N,L)\to\ldots.
\end{align*}

\textcolor{black}{Assume that }$P(1)=Ae_{1},\ldots,P(r)=Ae_{r}$ is
a complete list of the indecomposable projective modules of $A$ up
to isomorphism (where $e_{1}\ldots,e_{r}$ are primitive idempotents).
Denote by $S(i)=P(i)/\Rad P(i)$ the simple module corresponding to
$P(i)$.\textcolor{black}{{} The }\textcolor{black}{\emph{Cartan matrix
}}\textcolor{black}{of }$A$\textcolor{black}{{} is an }$r\times r$
matrix whose $(i,j)$ entry is the number of times that $S(i)$ appear
as a Jordan-Hölder factor of $P(j)$. This number is also equal to
$\dim e_{i}Ae_{j}$\textcolor{black}{.}

We denote by $\pd(M)$ the\emph{ projective dimension} of the $A$-module
$M$, which is the minimal $n$ for which $\Ext^{n+1}(M,N)=0$ for
every $A$-module $N$. The \emph{global dimension $\gd A$} of an
algebra $A$ is defined by 
\[
\gd A=\sup\{\pd(M)\mid M\text{ is an \ensuremath{A} module}\}
\]
and it is known that it is enough to take the supremum only on the
simple modules, that is,
\[
\gd A=\sup\{\pd(S)\mid S\text{ is a simple \ensuremath{A} module}\}.
\]

Two algebras $A$ and $B$ are called \emph{Morita equivalent} if
the category of all $A$-modules is equivalent to the category of
all $B$-modules. Morita equivalent algebras share many properties,
for instance they have the same global dimension and identical Cartan
matrices.

\textcolor{black}{The }\textcolor{black}{\emph{ordinary quiver}}\textcolor{black}{{}
}$Q$\textcolor{black}{{} of }$A$\textcolor{black}{{} is a directed graph
defined in the following way.} The vertices of $Q$ are in a one-to-one
correspondence with the irreducible representations of $A$ (up to
isomorphism). If $S(i)$ and $S(j)$ are two irreducible representations
of $A$ (identified with two vertices of the quiver), then the number
of arrows from $S(i)$ to $S(j)$ is

\[
\dim\Ext^{1}(S(i),S(j)).
\]
The quiver $Q$ of the algebra $A$ gives, in some sense, the generators
for $A$ in a generators and relations presentation. The exact explanation
is as follows. We denote by $Q^{\ast}$ the free category generated
by $Q$. $Q^{\ast}$ has precisely the same set of objects as $Q$
but its morphisms are paths in $Q$ (including a trivial path of length
$0$ for each object). Now we can form the algebra $\mathbb{C}Q^{\ast}$
which is called the \emph{path algebra} of $Q$. There exists an ideal
$I$ (satisfying some technical property called admissibility) such
that $\mathbb{C}Q^{\ast}/I$ is Morita equivalent to $A$. As usual,
we say that two elements $x,y\in\mathbb{C}Q^{\ast}$ are equivalent
(modulo $I$) if $x-y\in I$. It will be important to know few additional
facts about the quiver. For every object $a$ of $Q$, denote by $1_{a}$
the empty path of $a$. It is known that $\{1_{a}+I\mid a\in Q^{0}\}$
form a complete set of primitive orthogonal idempotents of $\mathbb{C}Q^{\ast}/I$.
Moreover, it is known that $1_{a}+I$ is not equivalent to $1_{b}+I$
if $a\ne b$. The projective module corresponding to $a\in Q^{0}$
is $P(a)=(\mathbb{C}Q^{\ast}/I)\cdot(1_{a}+I)$. It consists of all
equivalence classes of linear combination of paths that start at $a$.
It will be also important to understand how the Cartan matrix can
be seen inside the quiver presentation. For elements $a,b\in Q^{0}$,
denote by $S(b)$ the simple module that corresponds to $b$ and by
$V_{a,b}$ the $\mathbb{C}$-vector space spanned by the paths that
start at $a$ and end at $b$. The intersection $I\cap V_{a,b}$ is
a subspace of $V_{a,b}$. The number of times that $S(b)$ appears
as a Jordan-Hölder factor of $P(a)$ (i.e. the $(b,a)$ entry of the
Cartan matrix) is precisely the dimension of the quotient space $V_{a,b}/I\cap V_{a,b}$
(see \cite[Lemma 2.4 of Chapter III]{Assem2006}). In some sense this
is the number of paths from $a$ to $b$ modulo $I$.

Another important fact is that the global dimension of $A$ is bounded
above by the length of the longest path in $Q$ if $Q$ is acyclic.

\subsection{Complex group representations}

Let $G$ be a finite group. By Maschke's theorem, the complex group
algebra $\mathbb{C}G$ is a semisimple algebra. In particular, an
irreducible module $S$ is also an indecomposable projective module
so it is isomorphic to $\mathbb{C}Ge$ for some primitive idempotent
$e$. Moreover, it is known that if $e$ is a primitive idempotent
then there is an isomorphism of algebras $e\mathbb{C}Ge\simeq\mathbb{C}$.
We denote the trivial representation of any group $G$ by $\tr_{G}$
and the trivial representation of the symmetric group $S_{n}$ by
$\tr_{n}$. 

Let $H\subseteq G$ be a subgroup of $G$ and let $V$ ($U$) be a
$G$-module (respectively, an $H$-module). We denote by $\Res_{H}^{G}V$
and $\Ind_{H}^{G}U$ the \emph{restriction} and \emph{induction} representations,
respectively. Recall that 
\[
\Ind_{H}^{G}U=\mathbb{C}G\underset{\mathbb{C}H}{\otimes}U.
\]

For every $G$-representation $V$, we denote by $\chi_{V}$ its character.
For functions $\rho,\psi:G\to\mathbb{C}$, we denote by $\langle\rho,\psi\rangle$
the inner product 
\[
\langle\rho,\psi\rangle=\frac{1}{|G|}{\displaystyle \sum_{g\in G}}\rho(g)\overline{\psi(g)}.
\]
In order to simplify notation, we sometimes omit the $\chi$ and write
$V$ also for the character of the module $V$. For instance, we can
write Frobenius reciprocity as the following equality: 
\[
\langle\Ind_{H}^{G}V,U\rangle=\langle V,\Res_{H}^{G}U\rangle
\]

where $U$,$V$,$\Ind_{H}^{G}V$ and $\Res_{H}^{G}U$ are the respective
characters.

Assume that $G$ is acting on some finite set $X$. Denote by $\mathbb{C}X$
the vector space of all linear combinations of elements of $X$. $\mathbb{C}X$
is a $G$-representation in the natural way. A representation of this
form is called a \emph{permutation representation.} It is well known
that if $X_{1},\ldots,X_{r}$ are the orbits of this action then $\mathbb{C}X=\mathbb{C}X_{1}\oplus\cdots\oplus\mathbb{C}X_{r}$.
Now assume that $G$ is acting transitively on $X$ and\emph{ }let
$K$ the stabilizer of some $x\in X$. It is well known that $\mathbb{C}X\cong\Ind_{K}^{G}\tr_{K}$,
no matter which $x\in X$ is chosen.

We now consider the special case where $G=S_{n}$ is the symmetric
group. Proofs and more details on this case can be found in \cite{James1981,Sagan2001}.
Recall that an \emph{integer composition} of $n$ is a tuple $\lambda=[\lambda_{1},\ldots,\lambda_{k}]$
of non-negative integers such that $\lambda_{1}+\cdots+\lambda_{k}=n$
while an \emph{integer partition of $n$ }(denoted $\lambda\vdash n$)
is an integer composition such that $\lambda_{1}\geq\lambda_{2}\geq\cdots\geq\lambda_{k}>0$.
Note that $0$ has one partition, namely, the empty partition, denoted
by $\varnothing$. We can associate to any partition $\lambda$ a
graphical description called a \emph{Young diagram}, which is a table
with $\lambda_{i}$ boxes in its $i$-th row. For instance, the Young
diagram associated to the partition $[3,3,2,1]$ of $9$ is:

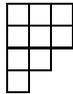
\ytableausetup{centertableaux}\ytableausetup{smalltableaux}
\[
\ydiagram{3,3,2,1}
\]

We will identify the two notions and regard integer partition and
Young diagram as synonyms. It is well known that irreducible representations
of $S_{n}$ are indexed by integer partitions of $n$. We denote the
irreducible representation associated to the partition $\lambda$
(also called its \emph{Specht module}) by $S^{\lambda}$. Explicit
description of $S^{\lambda}$ can be found in \cite[Section 2.3]{Sagan2001}.
A convenient abbreviation will be to write $[1^{k}]$ instead of $[\underset{k\text{ times}}{\underbrace{1,\ldots,1}}]$
and likewise $[2,1^{k}]$ for $[2,\underset{k\text{ times}}{\underbrace{1,\ldots,1}}]$
etc. 

We now turn to describe the Littlewood-Richardson branching rule that
will play a crucial role in the sequel. If we identify $S_{k}$ ($S_{r}$)
with the group of all permutations of $\{1,\ldots,k+r\}$ that leave
$\{k+1,\ldots,k+r\}$ (respectively, $\{1,\ldots,k\}$) fixed we can
view $S_{k}\times S_{r}$ as a subgroup of $S_{k+r}$. Given $\lambda\vdash k$
and $\delta\vdash r$, we denote by $S^{\lambda}\otimes S^{\delta}$
the outer tensor product of $S^{\lambda}$ and $S^{\delta}$ which
is a $\mbox{\ensuremath{S_{k}\times S_{r}}}$-representation. The
Littlewood\textendash Richardson rule gives the decomposition of $\Ind_{S_{k}\times S_{r}}^{S_{k+r}}(S^{\lambda}\otimes S^{\delta})$
into irreducible $S_{k+r}$-representations. In other words, if we
write this decomposition as 
\[
\Ind_{S_{k}\times S_{r}}^{S_{k+r}}(S^{\lambda}\otimes S^{\delta})=\bigoplus_{\gamma\vdash(k+r)}c_{\lambda,\delta}^{\gamma}S^{\gamma},
\]

it gives a combinatorial interpretation for the coefficients $c_{\lambda,\delta}^{\gamma}$
(called the Littlewood\textendash Richardson coefficients). For this
we have to introduce some more notions. First we generalize the notion
of a Young diagram. For $k\leq n$ and $r\leq s$, let $\lambda=[\lambda_{1},\cdots,\lambda_{r}]\vdash k$
and $\gamma=[\gamma_{1},\cdots,\gamma_{s}]\vdash n$ be partitions
such that $\lambda_{i}\leq\gamma_{i}$ for every $1\leq i\leq r$.
The \emph{skew diagram} $\gamma/\lambda$ is the diagram obtained
by erasing the diagram $\lambda$ from the diagram $\gamma$. For
instance, if $\lambda=[2,1]$ and $\gamma=[4,3,1]$ then $\gamma/\lambda$
is the skew diagram

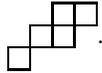
\ytableausetup{centertableaux}\ytableausetup{smalltableaux}
\[
\ydiagram{2+2,1+2,1}.
\]

A \emph{skew tableau} is a skew diagram whose boxes are filled with
numbers. We call the original diagram the \emph{shape} of the tableau.
Let $t$ be a skew tableau with $n$ boxes such that the number of
boxes with entry $i$ is $\delta_{i}$. The \emph{content} of $t$
is the composition $\delta=[\delta_{1},\ldots,\delta_{l}]$. We say
that a skew tableau is \emph{semi-standard} if its columns are increasing
and its rows are non-decreasing. For instance

\begin{align} \label{eq:exampleTableau}  \begin{ytableau}  \none & \none & 1 & 1 \\ \none & 2 & 3 \\ 2 \end{ytableau}\end{align}

is a semi-standard skew tableau of shape $[4,3,1]/[2,1]$ with content
$[2,2,1]$. The \emph{row word} of a skew tableau $t$ is the string
of numbers obtained by reading the entries of $t$ from right to left
and top to bottom. For instance, the row word of tableau \ref{eq:exampleTableau}
is $11322$. A string of numbers is called a \emph{lattice permutation}
if for every prefix of the string and for every number $i$, there
are no less occurrences of $i$ than occurrences of $i+1$. For instance,
the string $11322$ is not a lattice permutation since the prefix
$113$ contains one $3$ and no $2$'s. Now we can state the Littlewood-Richardson
rule (for proof see \cite[Theorem 2.8.13]{James1981}). 
\begin{thm}
\label{thm:Littlewood-Richardson rule} The Littlewood-Richardson
coefficient $c_{\lambda,\delta}^{\gamma}$ is the number of semi-standard
skew tableaux of shape $\gamma/\lambda$ with content $\delta$ whose
row word is a lattice permutation.
\end{thm}
The special case where $\delta=[r]=\tr_{S_{r}}$ is the trivial representation
of $S_{r}$ is called Pieri's rule. It is worth stating this special
case.
\begin{prop}[Pieri's rule]
\label{prop:Pieri'sRule} Let $\lambda\vdash k$ be a Young diagram.
Denote by $Y^{r}(\lambda)$ the set of Young diagrams obtained from
$\lambda$ by adding $r$ boxes with no two of them in the same column.
Then 
\[
\Ind_{S_{k}\times S_{r}}^{S_{k+r}}(S^{\lambda}\otimes\tr_{S_{r}})=\bigoplus_{\gamma\in Y^{r}(\lambda)}S^{\gamma}.
\]
\end{prop}

\section{\label{sec:CartanMatrixOfEICategoryAlgebra}Cartan matrix of an EI-category
algebra}
\begin{defn}
A category $\mathcal{D}$ is called an \emph{EI-category} if every
endomorphism is an isomorphism or, equivalently, if every endomorphism
monoid of $\mathcal{D}$ is a group.
\end{defn}
The goal of this section is to describe the Cartan matrix of the algebra
of a finite EI-category. By a description we mean that we want to
reduce the problem to the representation theory of the automorphism
groups. We will need this description in the next section only for
one specific EI-category, but giving the general case is quite the
same. Representation theory of EI-categories is a well-studied subject
(see \cite{Luck} or \cite[Chapter 1 Section 11]{Dieck}) and all
facts in this section appear in the literature or known as a folklore.
However, we sketch some of the proofs for the sake of completeness.
\begin{defn}
A category $\mathcal{D}$ is called \emph{skeletal} if no two objects
of $\mathcal{D}$ are isomorphic.
\end{defn}
Note that any category $\mathcal{D}$ is equivalent to some skeletal
category (which is unique up to isomorphism) called its \emph{skeleton}.
The skeleton of $\mathcal{D}$ is the full subcategory having one
object from every isomorphism class of $\mathcal{D}$. It is well
known that algebras of equivalent categories are Morita equivalent
\cite[Proposition 2.2]{Webb2007}, so they have the same Cartan matrix.
In particular, the algebra of an EI-category $\mathcal{D}$ and its
skeleton have the same Cartan matrix. Therefore, without any loss
of generality we can fix from now on $\mathcal{D}$ to be a finite
and skeletal EI-category and concentrate on finding the Cartan matrix
of $\mathcal{D}$. One simple but important observation on skeletal
EI-categories is that their objects are naturally ordered.
\begin{defn}
Let $\mathcal{D}$ be a skeletal EI-category. Define a relation $\leq_{\mathcal{D}}$
on $\mathcal{D}^{0}$ by $a\leq_{\mathcal{D}}b$ if $\mathcal{D}(a,b)\neq\varnothing$.
\end{defn}
\begin{lem}[{\cite[Page 170]{Luck}}]
The relation $\leq_{\mathcal{D}}$ is a partial order.
\end{lem}
Given $a\in\mathcal{D}^{0}$, denote $G_{a}=\mathcal{D}(a,a)$ to
be its automorphism group and $E_{a}=\{e_{1}^{a},\ldots,e_{m_{a}}^{a}\}$
a complete set of primitive orthogonal idempotents for $\mathbb{C}G_{a}$.
The following fact can be deduced from \cite[Lemma 9.31]{Luck} and
appears also in \cite[Proposition 2.3]{Webb}.
\begin{lem}
\label{lem:EquivalenceOfPrimitiveIdempotents}
\end{lem}
\begin{enumerate}
\item The set 
\[
\bigcup_{a\in D^{0}}E_{a}
\]
 is a complete set of primitive orthogonal idempotents for $\mathbb{C}\mathcal{D}$.
\item Two primitive idempotents $e_{i}^{a}$ and $e_{j}^{b}$ are equivalent
in $\mathbb{C}\mathcal{D}$ if and only if $a=b$ and they are equivalent
primitive idempotents of $\mathbb{C}G_{a}$.
\end{enumerate}
Let $a\in\mathcal{D}^{0}$ and denote by $I_{a}$ a set of indices
of all the primitive idempotents $E_{a}$ up to equivalence. Without
loss of generality we assume $I_{a}\cap I_{b}=\varnothing$ if $a\neq b$.
By \lemref{EquivalenceOfPrimitiveIdempotents} it is clear that 
\[
I=\bigcup_{a\in D^{0}}I_{a}
\]
contains indices of all the primitive idempotents ${\displaystyle \bigcup_{a\in D^{0}}}E_{a}$
up to equivalence. Later on we will have a natural set of indices
for the case we will be interested in so it will be very convenient
to work that way.

For every $i\in I_{a}$ we denote by $S^{i}=\mathbb{C}G_{a}e_{i}$
the simple module of $\mathbb{C}G_{a}$ corresponding to $e_{i}$.
We denote by $P(i)=\mathbb{C}\mathcal{D}e_{i}=\mathbb{C}\mathcal{D}{\displaystyle \underset{\mathbb{C}G_{a}}{\otimes}}\mathbb{C}G_{a}e_{i}$
the indecomposable projective corresponding to $e_{i}$ and by 
\[
S(i)=P(i)/\Rad P(i)
\]
the simple module of\emph{ $\mathbb{C}\mathcal{D}$ }corresponding
to $e_{i}$. By \lemref{EquivalenceOfPrimitiveIdempotents}, $\{P(i)\}_{i\in I}$
and $\{S(i)\}_{i\in I}$ are complete lists (up to isomorphism) of
the indecomposable projective and simple modules of $\mathbb{C}\mathcal{D}$.
Given $i,j\in I$ we want to know how many times $S(j)$ appears as
a Jordan-Hölder factor in $P(i)$.

For this we will need a basic fact about dual modules of groups. Recall
that for any $A$-module $M$, the dual module $D(M)=\Hom_{\mathbb{C}}(M,\mathbb{C})$
is a right $A$-module (or a left $A^{\op}$-module) defined by 
\[
(\varphi\cdot a)(v)=\varphi(av)
\]
for every $\varphi\in\Hom_{\mathbb{C}}(M,\mathbb{C})$, $a\in A$
and $v\in M$. It is well known that if $G$ is a group and $e$ is
a primitive idempotent then $e\mathbb{C}G\cong D(\mathbb{C}Ge)$ as
right $G$-modules (it is not difficult to check that they have the
same character). 

Now take $G,H$ to be two groups and let $\{e_{i}\}$ and $\{f_{i}\}$
be two complete sets of primitive orthogonal idempotents of $\mathbb{C}G$
and $\mathbb{C}H$ respectively. Let $M$ be a $\mathbb{C}G-\mathbb{C}H$
bimodule (or equivalently, a left $\mathbb{C}G\otimes(\mathbb{C}H)^{\op}$-module).
It is well-known that $e_{i}\otimes f_{j}$ is a primitive idempotent
of $\mathbb{C}G\otimes(\mathbb{C}H)^{\op}$ and the corresponding
simple module is $\mathbb{C}Ge_{i}{\displaystyle \otimes}f_{j}\mathbb{C}H\cong\mathbb{C}Ge_{i}{\displaystyle \otimes}D(\mathbb{C}Hf_{j})$.

The next step is to observe that the set $\mathbb{C}\mathcal{D}(a,b)$
of all morphisms from $a$ to $b$, has the structure of a $\mathbb{C}G_{b}-\mathbb{C}G_{a}$-bimodule
according to:
\[
(g_{1},g_{2})\cdot m=g_{1}mg_{2}.
\]
 Now we can give a description of the Cartan matrix in terms of the
representation theory of the automorphism groups.
\begin{prop}
\label{prop:CartanMatrixEICatFirstForm}Let $\mathcal{D}$ be a finite
and skeletal EI-category. Let $I={\displaystyle \bigcup_{a\in D^{0}}}I_{a}$
be a set of indices for the primitive idempotents up to equivalence
as described above. Take $i\in I_{a}$ and $j\in I_{b}$. The number
of times that $S(j)$ appears as a Jordan-Hölder factor in $P(i)$
is the number of times that $S^{j}\otimes D(S^{i})$ appears as an
irreducible constituent in the $G_{b}-G_{a}$-bimodule $\mathbb{C}\mathcal{D}(a,b)$.
\end{prop}
\begin{proof}
The number of times that $S(j)$ appears as a Jordan-Hölder factor
in $P(i)$, i.e., the $(j,i)$ entry of the Cartan matrix of $\mathbb{C}\mathcal{D}$
equals the dimension 
\[
\dim e_{j}\mathbb{C}\mathcal{D}e_{i}.
\]
Given $m\in\mathcal{D}^{1}$, it is clear that $e_{j}me_{i}=0$ unless
$m\in\mathcal{D}(a,b)$ hence we have 
\[
\dim e_{j}\mathbb{C}\mathcal{D}e_{i}.=\dim e_{j}\mathbb{C}\mathcal{D}(a,b)e_{i}
\]
and this is precisely the number of times that the simple bimodule
corresponding to $e_{j}\otimes e_{i}$, which is $\mathbb{C}G_{b}e_{j}{\displaystyle \otimes}e_{i}\mathbb{C}G_{a}\cong S^{j}\otimes D(S^{i})$,
appears in $\mathbb{C}\mathcal{D}(a,b)$ as a $\mathbb{C}G_{b}-\mathbb{C}G_{a}$-bimodule.
This finishes the proof.
\end{proof}
We will make another step in order to avoid explicit use of bimodules
(or opposite groups). If $M$ is a right $G$-module, we can also
regard $M$ as a left $G$-module with new action $\ast$ defined
by 
\[
g\ast m=m\cdot g^{-1}
\]
(where on the right hand side we use the right $G$-module action). 

This gives as an isomorphism between the category of right $G$-modules
(or left $G^{\op}$-modules) and the category of $G$-modules. This
is quite intuitive but a more accurate explanation can be given. We
can define a functor $\psi$ from the category of left $G^{\op}$-modules
to the category of $G$-modules in the following way. Consider the
function $\alpha_{G}:G\to G^{\op}$ which is defined by $\alpha_{G}(g)=g^{-1}$
(this is the usual natural isomorphism between $\id$ and $\op$ as
functors from the category of groups to itself). Note that if we think
of $G$ and $G^{\op}$ as one-object groupoids then $\alpha_{G}$
is actually a functor. A $G^{\op}$-module is just a functor $F$
from the group $G^{\op}$ viewed as a category to the category of
$\mathbb{C}$-vector spaces $\VS_{\mathbb{C}}$. $\psi$ is defined
on objects by $\psi(F)=F\circ\alpha_{G}$ and it is the identity function
on morphism (i.e. on the module homomorphism). It is not difficult
to check that $\psi$ is an isomorphism of categories. For example,
$\psi(D(M))$ is a \emph{left} module whose underlying set is again
$\Hom_{\mathbb{C}}(M,\mathbb{C})$ but now the action is

\[
(a\cdot\varphi)(v)=\varphi(a^{-1}v)
\]
for every $\varphi\in\Hom_{\mathbb{C}}(M,\mathbb{C})$, $a\in\mathbb{C}G$
and $v\in M$. We prefer do denote this module by $M^{\ast}$ rather
than $D(M)$.

In a similar way, any $\mathbb{C}(G\times H^{\op})$-module (or a
$\mathbb{C}G-\mathbb{C}H$ bi-module) can be regarded as an $G\times H$-module.
So we now think of $\mathbb{C}\mathcal{D}(a,b)$ as a $G_{b}\times G_{a}$-module
with the action
\[
(g_{1},g_{2})\cdot m=g_{1}mg_{2}^{-1}.
\]
\propref{CartanMatrixEICatFirstForm} can now be restated as follows.
\begin{prop}
\label{prop:CartanMatrixEICat}Let $\mathcal{D}$ be a finite and
skeletal EI-category. Let $I={\displaystyle \bigcup_{a\in D^{0}}}I_{a}$
be a set of indices to the primitive idempotents up to equivalence
as described above. Take $i\in I_{a}$ and $j\in I_{b}$. The number
of times that $S(j)$ appears as a Jordan-Hölder factor in $P(i)$
is the number of times that $S^{j}\otimes(S^{i})^{\ast}$ appears
as an irreducible constituent in the $G_{b}\times G_{a}$-module $\mathbb{C}\mathcal{D}(a,b)$.
\end{prop}
We remark that this description is very similar to descriptions of
the Cartan matrix that can be found in \cite[Definition 2.6]{Thiery2012Cartan}
and \cite[Corollary 7.28]{Steinberg2016b}.

\section{\label{sec:RepresentationTheoryOfPTnAndEn}Representation theory
of $\PT_{n}$ and $\E_{n}$}

Let $\PT_{n}$ denote the monoid of all partial functions on the set
$\{1,\ldots,n\}$. Also, denote by $\E_{n}$ the category defined
in the following way. The objects of $\E_{n}$ are the subsets of
$\{1,\ldots,n\}$ and for every two subsets $X$ and $Y$ the hom-set
$\E_{n}(X,Y)$ consists of all onto (total) functions with domain
$X$ and range $Y$. The following fact is proved in \cite[Proposition 3.2]{Stein2016}.
\begin{prop}
There is an isomorphism of algebras $\mathbb{C}\PT_{n}\simeq\mathbb{C}\E_{n}$.
\end{prop}
Therefore, from a representation theoretic point of view, $\PT_{n}$
and $\E_{n}$ has precisely the same properties. In particular, they
have the same global dimension. As mentioned above, computing this
global dimension is the goal of this paper. In this section we will
apply the results of the previous section for the case of $\E_{n}$.
Moreover, we will present some results on $\mathbb{C}\E_{n}$ that
were obtained in \cite{Stein2016}.

Given an object $X$ of $\E_{n}$, the hom-set $\E_{n}(X,X)$ consists
of all onto functions from $X$ to itself. So it is clear that the
endomorphism monoid $\E_{n}(X,X)$ is isomorphic to the group $S_{X}$
of all permutations of $X$. So $\E_{n}$ is an EI-category. Clearly,
two objects $X$ and $Y$ are isomorphic if and only if $|X|=|Y|$
so $\E_{n}$ is not skeletal. We will denote the skeleton of $\E_{n}$
by $\SE_{n}$. We can think of it as the category with object set
$\{0,\ldots,n\}$ such that the hom-set $\SE_{n}(r,k)$ contains all
the onto (total) functions from $\{1,\ldots,r\}$ to $\{1,\ldots,k\}$.
As mentioned in the previous section $\mathbb{C}\E_{n}$ is Morita
equivalent to $\mathbb{C}\SE_{n}$ so they have the same global dimension.
From now on we will concentrate in finding the global dimension of
$\mathbb{C}\SE_{n}$. The automorphism groups of $\SE_{n}$ are $S_{k}$
where $0\leq k\leq n$, It is well known that irreducible representations
of $S_{k}$ are parameterized by integer partitions of $k$, or equivalently,
by Young diagrams with $k$ boxes. So representations of $\SE_{n}$
are parameterized by Young diagrams with $k$ boxes where $0\leq k\leq n$.
Given such Young diagram $\alpha\vdash k$ we denote by $S^{\alpha}$
the Specht module corresponding to $\alpha$, which is an irreducible
representation of $\SE_{n}(k,k)\simeq S_{k}$. We denote by $S(\alpha)$
and $P(\alpha$) the simple and projective modules of $\mathbb{C}\SE_{n}$
corresponding to $\alpha$.

Recall that $S^{\alpha}\simeq(S^{\alpha})^{\ast}$ for every $\alpha\vdash k$
since they have the same character. Therefore, by \propref{CartanMatrixEICat}
we obtain:
\begin{cor}
\label{cor:CartanMatrixOfSE_n}Let $\alpha\vdash r$ and $\beta\vdash k$
be two Young diagrams. The number of times that $S(\beta)$ appears
as a Jordan-Hölder factor in $P(\alpha)$ is the number of times that
$S^{\beta}\otimes S^{\alpha}$ appears as an irreducible constituent
in the $S_{k}\times S_{r}$ module $\mathbb{C}\SE_{n}(r,k)$.
\end{cor}
\corref{CartanMatrixOfSE_n} gives a description of the Cartan matrix
of $\mathbb{C}\SE_{n}$ in terms of representations of the symmetric
group. On the other hand, given two Young diagrams $\alpha$ and $\beta$,
it is still very difficult, in general, to give an explicit combinatorial
description of the $(\beta,\alpha)$ entry of the Cartan matrix of
$\mathbb{C}\SE_{n}$. However, several observations are possible.
It is clear that the rows and columns of the Cartan matrix can be
indexed by Young diagram with $k$ boxes where $0\leq k\leq n$. We
will order them such that diagram with $r$ boxes appear before diagram
with $k$ boxes where $r<k$. Therefore we can think of the Cartan
matrix as a $(n+1)\times(n+1)$ block matrix where the $(i,j)$ block
contains pairs $(\beta,\alpha)$ of permutations such that $\beta\vdash i-1$
and $\alpha\vdash j-1$.
\begin{lem}
\label{lem:CartanMatrixIsUnitriangular}With ordering as just described,
the Cartan matrix of $\SE_{n}$ is block upper unitriangular.
\end{lem}
\begin{proof}
Let $\alpha\vdash r$ and $\beta\vdash k$ be two Young diagrams where
$r<k$. The hom-set $\SE_{n}(r,k)$ is empty so the by \corref{CartanMatrixOfSE_n}
it is clear that the $(\beta,\alpha)$ entry of the Cartan matrix
is $0$. So the elements below the diagonal are $0$. Now, regarding
a $(\beta,\alpha)$ entry where $\alpha\vdash r$, $\beta\vdash r$.
Denote by $e_{\alpha}$, $e_{\beta}$ two primitive idempotents corresponding
to the simple modules $S(\alpha)$,$S(\beta)$ respectively. We have
already seen that $e_{\alpha}$ and $e_{\beta}$ are also primitive
idempotents of $\mathbb{C}\SE_{n}(r,r)\simeq\mathbb{C}S_{r}$ corresponding
to the Specht modules $S^{\alpha}$,$S^{\beta}$ respectively. Therefore
the $(\beta,\alpha)$ entry of the Cartan matrix equals 
\begin{align*}
\dim e_{\beta}\mathbb{C}\SE_{n}e_{\alpha} & =\dim e_{\beta}\mathbb{C}\SE_{n}(r,r)e_{\alpha}\\
 & =\dim e_{\beta}S_{r}e_{\alpha}=\begin{cases}
\dim\mathbb{C}=1 & \text{if }\text{\ensuremath{\alpha}=\ensuremath{\beta}}\\
0 & \text{if }\alpha\neq\beta.
\end{cases}
\end{align*}
so the Cartan matrix is unitriangular as required.
\end{proof}
The $(\beta,\alpha)$ entry of the Cartan matrix where $\alpha\vdash(k+1)$
and $\beta\vdash k$ was found in \cite{Stein2016}. The following
proposition is a corollary of \cite[Theorem 3.4, Lemma 3.6 and Theorem 3.8]{Stein2016}.
\begin{prop}
\label{prop:CartanMatrixLength1}Let $\alpha\vdash(k+1)$ and $\beta\vdash k$
be two Young diagrams. The $(\beta,\alpha)$ entry of the Cartan matrix
of $\mathbb{C}\SE_{n}$ is the number of different ways that $\alpha$
can be constructed from $\beta$ by removing one box and then adding
two boxes but not in the same column.
\end{prop}
In the next section we will give a description of the $(\beta,\alpha)$
entry where $\alpha\vdash k+2$ and $\beta\vdash k$ but first we
would like to mention another fact about $\mathbb{C}\SE_{n}$. \propref{CartanMatrixLength1}
is actually also a combinatorial description for the quiver of $\mathbb{C}\SE_{n}$.
Since we will need this quiver for some observations we will state
this result. 
\begin{thm}
\cite[Theorem 3.8]{Stein2016}\label{thm:DescriptionOfTheQuiverOfPT_n}
The vertices in the quiver of $\mathbb{C}\SE_{n}$ are in \textcolor{black}{one-to-one
correspondence with} Young diagrams with $k$ boxes where $0\leq k\leq n$.
If $\alpha\vdash r$, \textup{$\beta\vdash k$} are two Young diagrams
such that $r\neq k+1$, then there are no arrows from $\alpha$ to
$\beta$. If $r=k+1$, then there are arrows from $\alpha$ to $\beta$
if we can construct $\alpha$ from $\beta$ by removing one box and
then adding two boxes but not in the same column. The number of arrows
is the number of different ways that this construction can be carried
out.
\end{thm}
\begin{example}
\label{exa:QuiverOfPT4}A full drawing of the quiver of $\mathbb{C}\PT_{4}$
is given in the following figure: 

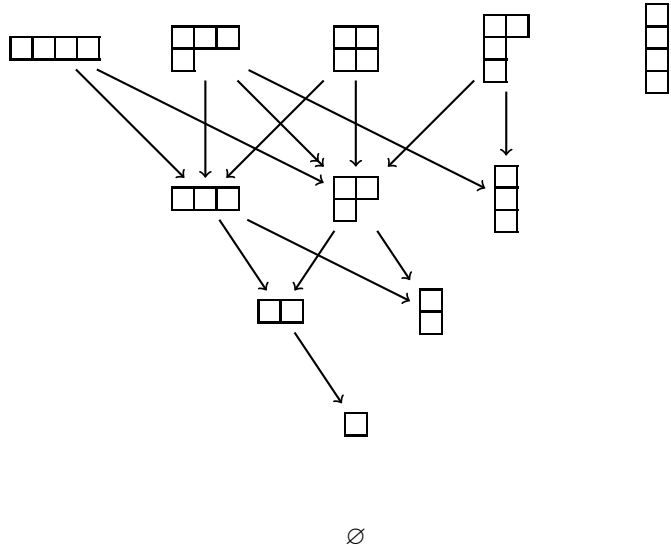
\ytableausetup{centertableaux}\ytableausetup{smalltableaux}\begin{center}\begin{tikzpicture}\path (-1,5) node (S4_1) {$\ydiagram{4}$};\path (1,5) node (S4_2) {$\ydiagram{3,1}$};\path (3,5) node (S4_3) {$\ydiagram{2,2}$};\path (5,5) node (S4_4) {$\ydiagram{2,1,1}$};\path (7,5) node (S4_5) {$\ydiagram{1,1,1,1}$}; \path (1,3) node (S3_1) {$\ydiagram{3}$}; \path (3,3) node (S3_2) {$\ydiagram{2,1}$}; \path (5,3) node (S3_3) {$\ydiagram{1,1,1}$}; \path (2,1.5) node (S2_1) {$\ydiagram{2}$}; \path (4,1.5) node (S2_2) {$\ydiagram{1,1}$}; \path (3,0) node (S1_1) {$\ydiagram{1}$}; \path (3,-1.5) node (S0_1) {$\varnothing$};\draw[thick,->] (S4_1)--(S3_1) ;\draw[thick,->] (S4_2)--(S3_1) ;\draw[thick,->] (S4_3)--(S3_1) ; \draw[thick,->] (S4_1)--(S3_2) ;\draw[thick,->>] (S4_2)--(S3_2) ;\draw[thick,->] (S4_3)--(S3_2) ;\draw[thick,->] (S4_4)--(S3_2) ; \draw[thick,->] (S4_2)--(S3_3) ; \draw[thick,->] (S4_4)--(S3_3) ; \draw[thick,->] (S3_1)--(S2_1) ; \draw[thick,->] (S3_2)--(S2_1) ; \draw[thick,->] (S3_1)--(S2_2) ; \draw[thick,->] (S3_2)--(S2_2) ; \draw[thick,->] (S2_1)--(S1_1) ;\end{tikzpicture}\end{center}
\end{example}

\section{\label{sec:TheSecondBlock}The second block superdiagonal of the
Cartan matrix}

In this section we will give an explicit description for the second
block superdiagonal of the Cartan matrix of $\mathbb{C}\SE_{n}$.
In other words, given $\alpha\vdash(k+2)$ and $\beta\vdash k$ we
will give a combinatorial interpretation for the number of times that
$S(\beta)$ appears as a Jordan-Hölder factor of $P(\alpha)$. \corref{CartanMatrixOfSE_n}
implies that we will have to understand better the action of $S_{k}\times S_{k+2}$
on $\mathbb{C}\SE_{n}(k+2,k)$.
\begin{defn}
Let $\theta$ be an equivalence relation on some finite set $X$.
The \emph{integer partition of $\theta$} is the integer partition
whose elements are the sizes of the equivalence classes of $\theta$.
We denote this integer partition by $I(\theta)$.

Let $f:X\to Y$ be a function. Recall that the kernel $\ker f$ is
the equivalence relation on $X$ defined by $a_{1}$ is equivalent
to $a_{2}$ if $f(a_{1})=f(a_{2})$. Consider some function $f\in\mathbb{C}\SE_{n}(k+2,k)$
where $k\geq2$. Since the kernel $\ker f$ partitions $\{1,\ldots,k+2\}$
into $k$ classes, the integer partition corresponding to $\ker f$
can be either $[3,1^{k-1}]$ or $[2^{2},1^{k-2}]$ and no other option
is possible. We claim that these two options give precisely the orbits
of our action.
\end{defn}
\begin{lem}
\label{lem:OrbitsOfActionSecondSubdiagonal}Let $k\geq2$. The sets
\[
O_{1}=\{f\in\mathbb{C}\SE_{n}(k+2,k)\mid I(\ker f)=\ensuremath{[3,1^{k-1}]}\}
\]
\[
O_{2}=\{f\in\mathbb{C}\SE_{n}(k+2,k)\mid I(\ker f)=\ensuremath{\ensuremath{[2^{2},1^{k-2}]}}\}
\]
form precisely the orbits of $\SE_{n}(k+2,k)$ under the action of
$S_{k}\times S_{k+2}$ described above. 
\end{lem}
\begin{proof}
As mentioned above it is clear that $\SE_{n}(k+2,k)=O_{1}\cup O_{2}$.
We want to prove that these are indeed orbits. Define two functions
\[
\kappa_{1},\kappa_{2}:\{1,\ldots,k+2\}\to\{1,\ldots,k\}
\]
 by
\[
\kappa_{1}(i)=\begin{cases}
i & i\leq k\\
k & i\in\{k+1,k+2\}
\end{cases}
\]
and
\[
\kappa_{2}(i)=\begin{cases}
i & i\leq k-1\\
k-1 & i=k\\
k & i\in\{k+1,k+2\}
\end{cases}.
\]
Clearly, $\kappa_{1}\in O_{1}$ and $\kappa_{2}\in O_{2}$. Now take
some other $f\in O_{1}$ and denote by $j_{1},j_{2},j_{3}$ the three
elements such that $f(j_{1})=f(j_{2})=f(j_{3})$. We can take any
$\pi\in S_{k+2}$ which satisfies
\[
\pi^{-1}(j_{1})=k,\quad\pi^{-1}(j_{2})=k+1,\quad\pi^{-1}(j_{3})=k+2
\]
and define $\sigma\in S_{k}$ to be the restriction of $f\pi$ to
$\{1,\ldots,k\}$. It is now easy to check that $f=\sigma\kappa_{1}\pi^{-1}$
so $f$ is in the same orbit as $\kappa_{1}$. Next, take some $g\in O_{2}$
and denote by $\{j_{1},j_{2}\}$ and $\{j_{3},j_{4}\}$ two (distinct)
sets such that $g(j_{1})=g(j_{2})$ and $g(j_{3})=g(j_{4})$. We can
take $\pi\in S_{k+2}$ to be any permutation that satisfies 
\begin{align*}
\pi^{-1}(j_{1}) & =k-1,\quad\pi^{-1}(j_{2})=k\\
\pi^{-1}(j_{3}) & =k+1,\quad\pi^{-1}(j_{4})=k+2
\end{align*}
and define $\sigma\in S_{k}$ by 
\[
\sigma(i)=\begin{cases}
g\pi(i) & i\neq k\\
g(j_{3}) & i=k
\end{cases}.
\]
Again, it is easy to see that $g=\sigma\kappa_{2}\pi^{-1}$ so $g$
is in the same orbit as $\kappa_{2}$. It is only left to show that
$\kappa_{1}$ and $\kappa_{2}$ are not in the same orbit. Indeed,
for every $\pi\in S_{k+2}$ and $\sigma\in S_{k}$ we have that the
elements $\pi(k),\pi(k+1),\pi(k+2)$ are in the same class of the
kernel of $\sigma\kappa_{1}\pi^{-1}$ so the corresponding partition
of $\sigma\kappa_{1}\pi^{-1}$ is also $[3,1^{k-1}]$ hence $\sigma\kappa_{1}\pi^{-1}\neq\kappa_{2}$.
This finishes the proof.
\end{proof}
Now we know that if $k\geq2$ the $S_{k}\times S_{k+2}$-module $\mathbb{C}\SE_{n}(k+2,k)$
decomposes into the direct sum of $\mathbb{C}O_{1}$ and $\mathbb{C}O_{2}$.
We will compute the multiplicity of $S^{\beta}\otimes S^{\alpha}$
as an irreducible constituent in $\mathbb{C}O_{1}$ and in $\mathbb{C}O_{2}$
separately. Since $\mathbb{C}O_{1}$ ($\mathbb{C}O_{2}$) is a permutation
representation of a transitive $S_{k}\times S_{k+2}$ action, $\mathbb{C}O_{1}$
(respectively, $\mathbb{C}O_{2}$) is isomorphic to $\Ind_{K}^{S_{k}\times S_{k+2}}\tr_{K}$
where $K$ is the stabilizer of some $f\in O_{1}$ (respectively,
$f\in O_{2}$). We start by investigating the action on $O_{1}$.
We will continue to use $\kappa_{1}$ and $\kappa_{2}$ that were
defined in \lemref{OrbitsOfActionSecondSubdiagonal}.
\begin{lem}
Consider in the usual way $S_{k-1}$ and $S_{k-1}\times S_{3}$ as
subgroups of $S_{k}$ and $S_{k+2}$ respectively. The stabilizer
of $\kappa_{1}\in O_{1}$ is 
\[
K_{1}=\{(\rho,\rho\tau)\mid\rho\in S_{k-1},\tau\in S_{3}\}\simeq S_{k-1}\times S_{3}.
\]
\end{lem}
\begin{proof}
Assume that $\sigma\kappa_{1}\pi^{-1}=\kappa_{1}$ for some $\sigma\in S_{k}$
and $\pi\in S_{k+2}$. Take some $i<k$. If $\pi^{-1}(i)\in\{k,k+1,k+2\}$
then there exists some $j\in\{1,\ldots,k+2\}$ (not equal to $i$)
such that $\kappa_{1}\pi^{-1}(i)=\kappa_{1}\pi^{-1}(j)$ and hence
\[
\kappa_{1}(i)=\sigma\kappa_{1}\pi^{-1}(i)=\sigma\kappa_{1}\pi^{-1}(j)=\kappa_{1}(j)
\]
which contradicts the definition of $\kappa_{1}$. So $\pi^{-1}$
must permute $\{1,\ldots,k-1\}$ and $\{k,k+1,k+2\}$ separately.
So there are some $\rho\in S_{k-1}$ and $\tau\in S_{3}$ such that
$\pi=\rho\tau$ and hence $\pi^{-1}=\rho^{-1}\tau^{-1}$(we think
of $S_{k-1}\times S_{3}$ as a subgroup of $S_{k+2}$ in the usual
way). Since $\kappa_{1}$ is the identity on $\{1,\ldots,k-1\}$,
it is clear that the restriction of $\sigma$ on $\{1,\ldots,k-1\}$
is $\rho$. This clearly implies that $\sigma(k)=k$ so $\sigma=\rho$
(considered as an element of $S_{k}$ by the usual embedding of $S_{k-1}$
in $S_{k}$). It is also easy to see that for every $\rho\in S_{k-1}$
and $\tau\in S_{3}$ we have that $\rho\kappa_{1}(\rho\tau)^{-1}=\kappa_{1}$.
We conclude that 
\[
K_{1}=\Stab(\kappa_{1})=\{(\rho,\rho\tau)\mid\rho\in S_{k-1},\quad\tau\in S_{3}\}\cong S_{k-1}\times S_{3}
\]
as required.
\end{proof}
\begin{lem}
\label{lem:MultiplicityInO1AsModule}Let $\alpha\vdash(k+2)$, $\beta\vdash k$
and assume $k\geq1$. The multiplicity of $S^{\beta}\otimes S^{\alpha}$
as an irreducible constituent in $\mathbb{C}O_{1}\simeq\Ind_{K_{1}}^{S_{k}\times S_{k+2}}\tr_{K_{1}}$
equals the multiplicity of $S^{\alpha}$ as an irreducible constituent
in the $S_{k+2}$-module 
\[
\Ind_{S_{k-1}\times S_{3}}^{S_{k+2}}(\Res_{S_{k-1}}^{S_{k}}(S^{\beta})\otimes\tr_{3}).
\]
\end{lem}
\begin{proof}
The multiplicity of $S^{\beta}\otimes S^{\alpha}$ in $\mathbb{C}O_{1}$
can be expressed by the inner product of characters: 
\[
\langle S^{\beta}\otimes S^{\alpha},\Ind_{K_{1}}^{S_{k}\times S_{k+2}}\tr_{K_{1}}\rangle
\]

(recall that in order to simplify notation, we use the same notation
for the representation and its character). Using Frobenius reciprocity,
we can see that 
\begin{align*}
\langle S^{\beta}\otimes S^{\alpha},\Ind_{K_{1}}^{S_{k}\times S_{k+2}}\tr_{K_{1}}\rangle & =\langle\Res_{K_{1}}^{S_{k}\times S_{k+2}}(S^{\beta}\otimes S^{\alpha}),\tr_{K_{1}}\rangle\\
 & =\frac{1}{|K_{1}|}\sum_{(\rho,\rho\tau)\in K_{1}}S^{\beta}\otimes S^{\alpha}((\rho,\rho\tau))\\
 & =\frac{1}{|K_{1}|}\sum_{(\rho,\tau)\in S_{k-1}\times S_{3}}S^{\beta}(\rho)S^{\alpha}(\rho\tau).
\end{align*}

This equals 
\begin{align*}
\frac{1}{|K_{1}|}\sum_{(\rho,\tau)\in S_{k-1}\times S_{3}}S^{\beta}(\rho)S^{\alpha}(\rho\tau) & =\frac{1}{|K_{1}|}\sum_{(\rho,\tau)\in S_{k-1}\times S_{3}}S^{\alpha}(\rho\tau)S^{\beta}(\rho)\tr_{3}(\tau)\\
 & =\langle\Res_{K_{1}}^{S_{k+2}}S^{\alpha},\Res_{S_{k-1}}^{S_{k}}(S^{\beta})\otimes\tr_{3}\rangle
\end{align*}
where $\tr_{3}$ is the trivial representation of $S_{3}$. Again,
using Frobenius reciprocity this equals 
\[
\langle S^{\alpha},\Ind_{S_{k-1}\times S_{3}}^{S_{k+2}}(\Res_{S_{k-1}}^{S_{k}}(S^{\beta})\otimes\tr_{3})\rangle.
\]
\end{proof}
Using Pieri's rule (\propref{Pieri'sRule}) we obtain the following
corollary.
\begin{cor}
\label{cor:MultiplicityInO1}Let $\alpha\vdash(k+2)$, $\beta\vdash k$
and assume $k\geq1$. The $S_{k}\times S_{k+2}$-module $S^{\alpha}\otimes S^{\beta}$
appears as an irreducible constituent in $\mathbb{C}O_{1}$ if $\beta$
can be obtained from $\alpha$ by removing one box and then adding
three, but no two in the same column. The multiplicity is the number
of different ways that this construction can be carried out.
\end{cor}
\begin{rem}
\label{rem:RemarkOnLevel3}Note that the decomposition to $O_{1}$
and $O_{2}$ given in \lemref{OrbitsOfActionSecondSubdiagonal} holds
only if $k\geq2$. Now, consider the case $k=1$. In this case the
$S_{1}\times S_{3}\simeq S_{3}$ action on $\SE_{n}(3,1)$ is transitive.
Actually, it is isomorphic to $O_{1}$. So this case it is completely
described by \corref{MultiplicityInO1}. In this case it is very easy
to describe the situation. $S_{3}$ has $3$ representations that
correspond to the Young diagrams $[3]$, $[2,1]$ and $[1^{3}]$.
$S_{1}$ has only the trivial representation $[1]$. If we remove
one box from $[1]$ and add three but no two of them in the same column,
we can obtain only $[3]$ and only in one way. So $S([1])$ appears
as a Jordan-Hölder factor of $P([3])$ with multiplicity $1$ and
doesn't appear in $P([2,1])$ and $P([1^{3}])$.
\end{rem}
Now we turn to investigate the decomposition of $\mathbb{C}O_{2}$
(for $k\geq2$) into irreducible modules. The idea is similar to what
we did with $\mathbb{C}O_{1}$ but the details are more complicated.
We start with the following observation.
\begin{rem}
\label{rem:ElementsinTheDihedral}Consider the dihedral group $D_{4}$
as the subgroup of $S_{4}$ with generators $a=(12)$ and $b=(13)(24)$.
Since $D_{4}$ can be presented by 
\[
\langle x,y\mid x^{2}=y^{2}=1,\quad(xy)^{4}=1\rangle
\]
 it is easy to check that the function $\nu:D_{4}\to S_{2}$ defined
by $\nu(a)=\id$ and $\nu(b)=(12)$ is a group homomorphism. Now,
let $\kappa:\{1,2,3,4\}\to\{1,2\}$ be defined by $\kappa(1)=\kappa(2)=1$
and $\kappa(3)=\kappa(4)=2$. Assume $\tau\in S_{4}$ and $\tau^{\prime}\in S_{2}$
are functions such that $\tau^{\prime}\kappa\tau=\kappa$. It is easy
to see that this implies that $\tau\in D_{4}$ and $\tau^{\prime}=\nu(\tau)$.
Note that we can give a ``geometric'' interpretation for $\nu$.
Consider the standard action of $D_{4}$ on a square.

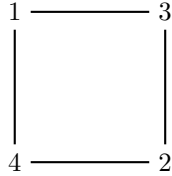
\begin{center}\begin{tikzpicture}\path (-2,0) node (4) {$4$}; \path (0,0) node (2) {$2$}; \path (-2,2) node (1) {$1$}; \path (0,2) node (3) {$3$}; \draw[thick] (1)--(3); \draw[thick] (3)--(2); \draw[thick] (2)--(4); \draw[thick] (4)--(1); \end{tikzpicture}  \end{center}

The kernel of $\nu$ is precisely the set of elements that keep each
opposite pair of corners occupied by the same pair of numbers. In
others words, these are the elements that keep the upper left and
the bottom right corners occupied by $\{1,2\}$ and the other two
corners occupied by $\{3,4\}$. We will abbreviate and say that these
elements are \emph{keeping corners.}
\end{rem}
\begin{lem}
Assume $k\geq2$. We consider $D_{4}$ as a subgroup of $S_{4}$ as
described in \remref{ElementsinTheDihedral}. Therefore, we think
of $S_{k-2}\times D_{4}$ ($S_{k-2}\times S_{2}$) as a subgroup of
$S_{k+2}$ (respectively, $S_{k}$). The stabilizer of $\kappa_{2}\in O_{2}$
is 
\[
K_{2}=\{(\rho\nu(\tau),\rho\tau)\mid\rho\in S_{k-2},\quad\tau\in D_{4}\}\simeq S_{k-2}\times D_{4}.
\]
\end{lem}
\begin{proof}
Assume that $\sigma\kappa_{2}\pi^{-1}=\kappa_{2}$ for some $\sigma\in S_{k}$
and $\pi\in S_{k+2}$. As

\[
\sigma\kappa_{2}\pi^{-1}(k-1)=\kappa_{2}(k-1)=\kappa_{2}(k)=\sigma\kappa_{2}\pi^{-1}(k)
\]
and since $\pi$ and $\sigma$ are permutations, it is clear that
$\pi^{-1}(k-1)$ and $\pi^{-1}(k)$ are in the same kernel class of
$\kappa_{2}$. This implies that 
\[
\pi^{-1}(k-1),\pi^{-1}(k)\in\{k-1,k,k+1,k+2\}
\]
and likewise 
\[
\pi^{-1}(k+1),\pi^{-1}(k+2)\in\{k-1,k,k+1,k+2\}.
\]
So $\pi^{-1}$ must permute $\{1,\ldots,k-2\}$ and $\{k-1,k,k+1,k+2\}$
separately. So there are some $\rho\in S_{k-2}$ and $\tau\in S_{4}$
such that $\pi=\rho\tau$ hence $\pi^{-1}=\rho^{-1}\tau^{-1}$(we
think of $S_{k-2}\times S_{4}$ as a subgroup of $S_{k+2}$ in the
usual way). Since $\kappa_{2}$ is the identity on $\{1,\ldots,k-2\}$,
it is clear that the restriction of $\sigma$ on $\{1,\ldots,k-2\}$
is $\rho$. Now, denote by $\widetilde{\kappa_{2}}$ and by $\tilde{\sigma}$
the restrictions of $\kappa_{2}$ and $\sigma$ on $\{k-1,k,k+1,k+2\}$
and $\{k-1,k\}$, respectively. Since $\tilde{\sigma}\widetilde{\kappa_{2}}\tau^{-1}=\widetilde{\kappa_{2}}$
we know by \remref{ElementsinTheDihedral} that $\tau\in D_{4}$ (with
the obvious identification between $\{k-1,k,k+1,k+2\}$ and $\{1,2,3,4\}$)
and $\tilde{\sigma}=\nu(\tau^{-1})=(\nu(\tau))^{-1}=\nu(\tau)$ (note
that $\nu(\tau)\in S_{2}$ so it is the inverse of itself). In conclusion
we obtain that 
\[
K_{2}=\Stab(\kappa_{2})=\{(\rho\nu(\tau),\rho\tau)\mid\rho\in S_{k-2},\quad\tau\in D_{4}\}\simeq S_{k-2}\times D_{4}
\]
as required.
\end{proof}
Now we want to find out what is the multiplicity of $S^{\beta}{\displaystyle \otimes S^{\alpha}}$
as an irreducible constituent in the $S_{k}\times S_{k+2}$ module
$\mathbb{C}O_{2}$ which is isomorphic to 
\[
\Ind_{K_{2}}^{S_{k}\times S_{k+2}}\tr_{K_{2}}.
\]
The idea is similar to what we did with $\mathbb{C}O_{1}$ but here
the situation is more complicated. We will have to start with some
more observations.

Let $W$ be an $S_{2}$-representation. We will denote by $\overline{W}$
the inflation of $W$ to a $D_{4}$-representation along the homomorphism
$\nu:D_{4}\to S_{2}$. Likewise, if $W$ is a $G\times S_{2}$-representation
we will denote by $\overline{W}$ its inflation into a $G\times D_{4}$-representation
along the homomorphism $\id_{G}\times\nu$. It is not difficult to
describe explicitly this inflation but we will do so only after the
next lemma.
\begin{lem}
\label{lem:MultiplicityInO2}Let $\alpha\vdash(k+2)$, $\beta\vdash k$
and assume $k\geq2$. The multiplicity of $S^{\beta}\otimes S^{\alpha}$
as an irreducible constituent in $\mathbb{C}O_{2}$ equals the multiplicity
of $S^{\alpha}$ as an irreducible constituent in the $S_{k+2}$-module
\[
\Ind_{S_{k-2}\times D_{4}}^{S_{k+2}}\overline{\Res_{S_{k-2}\times S_{2}}^{S_{k}}S^{\beta}}.
\]
\end{lem}
\begin{proof}
By Frobenius reciprocity
\begin{align*}
\langle S^{\beta}\otimes S^{\alpha},\Ind_{K_{2}}^{S_{k}\times S_{k+2}}\tr_{K_{2}}\rangle & =\langle\Res_{K_{2}}^{S_{k}\times S_{k+2}}(S^{\beta}\otimes S^{\alpha}),\tr_{K_{2}}\rangle\\
 & =\frac{1}{|K_{2}|}\sum_{(\rho\nu(\tau),\rho\tau)\in K_{2}}S^{\beta}\otimes S^{\alpha}((\rho\nu(\tau),\rho\tau))\\
 & =\frac{1}{|K_{2}|}\sum_{(\rho,\tau)\in S_{k-2}\times D_{4}}S^{\beta}(\rho\nu(\tau))S^{\alpha}(\rho\tau)
\end{align*}

Again, we want to express this sum as the inner product of two $\mbox{\ensuremath{S_{k-2}\times D_{4}}-representations}$.
First observe that 
\[
S^{\beta}(\rho\nu(\tau))=\overline{\Res_{S_{k-2}\times S_{2}}^{S_{k}}S^{\beta}}(\rho,\tau)
\]
where here we inflate the $S_{k-2}\times S_{2}$-representation 
\[
\Res_{S_{k-2}\times S_{2}}^{S_{k}}S^{\beta}
\]
into a $S_{k-2}\times D_{4}$ representation. Moreover, it is clear
that 
\[
S^{\alpha}(\rho\tau)=\Res_{S_{k-2}\times D_{4}}^{S_{k+2}}S^{\alpha}(\rho\tau).
\]
Therefore, the above expression equals 
\[
\langle\Res_{S_{k-2}\times D_{4}}^{S_{k+2}}S^{\alpha},\overline{\Res_{S_{k-2}\times S_{2}}^{S_{k}}S^{\beta}}\rangle
\]
and by Frobenius reciprocity, this equals
\[
\langle S^{\alpha},\Ind_{S_{k-2}\times D_{4}}^{S_{k+2}}\overline{\Res_{S_{k-2}\times S_{2}}^{S_{k}}S^{\beta}}\rangle
\]
as required.
\end{proof}
By \lemref{MultiplicityInO1AsModule} and \lemref{MultiplicityInO2}
we obtain the following corollary.
\begin{cor}
\label{lem:CombinedO1O2}Let $k\geq2$. The number of times that $S(\beta)$
appears as a Jordan-Hölder factor of $P(\alpha)$, which is the multiplicity
of $S^{\beta}\otimes S^{\alpha}$ as an irreducible constituent in
$\mathbb{C}\SE_{n}(k+2,k)$ equals the multiplicity of $S^{\alpha}$
as an irreducible constituent in the $S_{k+2}$-module
\[
\Ind_{S_{k-1}\times S_{3}}^{S_{k+2}}(\Res_{S_{k-1}}^{S_{k}}(S^{\beta})\otimes\tr_{3})\oplus\Ind_{S_{k-2}\times D_{4}}^{S_{k+2}}\overline{\Res_{S_{k-2}\times S_{2}}^{S_{k}}S^{\beta}}.
\]
\end{cor}
Now we want to explain how the above multiplicity can, in principle,
be computed. The representation
\[
\Ind_{S_{k-1}\times S_{3}}^{S_{k+2}}(\Res_{S_{k-1}}^{S_{k}}(S^{\beta})\otimes\tr_{3})
\]
can be computed using standard Branching rules and Pieri's rule. However,
it is more difficult to compute
\[
\Ind_{S_{k-2}\times D_{4}}^{S_{k+2}}\overline{\Res_{S_{k-2}\times S_{2}}^{S_{k}}S^{\beta}}.
\]
For this we will have to investigate more carefully the inflation
we are doing. $S_{2}$ has only two representation, the trivial representation
$\tr_{2}$ and the sign representation $\sgn_{2}$. It is obvious
that $\overline{\tr_{2}}=\tr_{D_{4}}$. Now, $\overline{\sgn_{2}}$
is also a one-dimensional representation. By \remref{ElementsinTheDihedral}
we can describe it as a representation that sends the four permutations
that keep corners to $1$ and the other four elements to $-1$. 

Now we want to understand what happens when we induce these representations
to $S_{4}$. In other words, we want to find $\Ind_{D_{4}}^{S_{4}}\tr_{D_{4}}$
and $\Ind_{D_{4}}^{S_{4}}\overline{\sgn_{2}}$.
\begin{lem}
\label{lem:DecompositionOfInductionOfTrD4}The decomposition into
irreducible representations of the $S_{4}$-representation $\Ind_{D_{4}}^{S_{4}}\tr_{D_{4}}$
is: 
\[
\Ind_{D_{4}}^{S_{4}}\tr_{D_{4}}\simeq S^{[4]}\oplus S^{[2,2]}
\]
\end{lem}
\begin{proof}
$\Ind_{D_{4}}^{S_{4}}\tr_{D_{4}}$ is a permutation representation
of the action of $S_{4}$ on the cosets $S_{4}/D_{4}$. Since this
action is transitive, we know that the multiplicity of the trivial
representation $\tr_{4}=S^{[4]}$ is $1$ (see \cite[Corollary B.11]{Steinberg2016b}).
Now, the multiplicity of $S^{[2,2]}$ as an irreducible constituent
of $\Ind_{D_{4}}^{S_{4}}\tr_{D_{4}}$ is 
\[
\langle S^{[2,2]},\Ind_{D_{4}}^{S_{4}}\tr_{D_{4}}\rangle=\langle\Res_{D_{4}}^{S_{4}}\left(S^{[2,2]}\right),\tr_{D_{4}}\rangle
\]
Note that $\Res_{D_{4}}^{S_{4}}\left(S^{[2,2]}\right)$ is just the
restriction of the character of $S^{[2,2]}$ as an $S_{4}$-representation.
This character is given in the following table:

\begin{center}

\begin{tabular}{|c|c|c|c|}
\hline 
$\id$ & $(12),(34)$ & $(12)(34),(13)(24),(14)(23)$ & $(1324),(1423)$\tabularnewline
\hline 
\hline 
$2$ & $0$ & $2$ & $0$\tabularnewline
\hline 
\end{tabular}

\end{center}

It is easy to calculate that 
\[
\langle\Res_{D_{4}}^{S_{4}}\left(S^{[2,2]}\right),\tr_{D_{4}}\rangle=1.
\]
Now, note that
\[
\dim S^{[4]}=1,\quad\dim S^{[2,2]}=2
\]
and 
\[
\dim\Ind_{D_{4}}^{S_{4}}\tr_{D_{4}}=[S_{4}:D_{4}]\cdot\dim\tr_{D_{4}}=3\cdot1=3.
\]
Therefore, by considering the dimensions we must have that 
\[
\Ind_{D_{4}}^{S_{4}}\tr_{D_{4}}\simeq S^{[4]}\oplus S^{[2,2]}
\]

as required.
\end{proof}
\begin{lem}
The decomposition into irreducible representations of the $S_{4}$-representation
$\Ind_{D_{4}}^{S_{4}}\overline{\sgn_{2}}$ is 
\[
\Ind_{D_{4}}^{S_{4}}\overline{\sgn_{2}}\simeq S^{[3,1]}.
\]
\end{lem}
\begin{proof}
The multiplicity of $S^{[3,1]}$ as an irreducible constituent of
$\Ind_{D_{4}}^{S_{4}}\overline{\sgn_{2}}$ is 
\[
\langle S^{[3,1]},\Ind_{D_{4}}^{S_{4}}\overline{\sgn_{2}}\rangle=\langle\Res_{D_{4}}^{S_{4}}\left(S^{[3,1]}\right),\overline{\sgn_{2}}\rangle
\]
Note that $\Res_{D_{4}}^{S_{4}}\left(S^{[3,1]}\right)$ is just the
restriction of the character of $S^{[3,1]}$ as an $S_{4}$-representation.
This character is given in the following table: 

\begin{center}

\begin{tabular}{|c|c|c|c|c|c|c|c|}
\hline 
$\id$ & $(12)$ & $(34)$ & $(12)(34)$ & $(13)(24)$ & $(14)(23)$ & $(1324)$ & $(1423)$\tabularnewline
\hline 
\hline 
$3$ & $1$ & $1$ & $-1$ & $-1$ & $-1$ & $-1$ & $-1$\tabularnewline
\hline 
\end{tabular}

\end{center}

By \remref{ElementsinTheDihedral} it is clear that the character
$\overline{\sgn_{2}}$ is given in the following table:

\begin{center}

\begin{tabular}{|c|c|c|c|c|c|c|c|}
\hline 
$\id$ & $(12)$ & $(34)$ & $(12)(34)$ & $(13)(24)$ & $(14)(23)$ & $(1324)$ & $(1423)$\tabularnewline
\hline 
\hline 
$1$ & $1$ & $1$ & $1$ & $-1$ & $-1$ & $-1$ & $-1$\tabularnewline
\hline 
\end{tabular}

\end{center}

Note that the first four permutations are the keeping corners elements.

Now it is easy to calculate that 
\[
\langle\Res_{D_{4}}^{S_{4}}\left(S^{[3,1]}\right),\overline{\sgn_{2}}\rangle=1.
\]
As before, we can finish with dimension considerations. We have that
\[
\dim\Ind_{D_{4}}^{S_{4}}\overline{\sgn_{2}}=3
\]
and we know that
\[
\dim S^{[3,1]}=3
\]

so we must have

\[
\Ind_{D_{4}}^{S_{4}}\overline{\sgn_{2}}\simeq S^{[3,1]}
\]
as required.
\end{proof}
\begin{rem}
\label{rem:ComputationOfMultInO2}Now we are, in principle, able to
compute the expression 
\[
\Ind_{S_{k-2}\times D_{4}}^{S_{k+2}}\overline{\Res_{S_{k-2}\times S_{2}}^{S_{k}}S^{\beta}}
\]
 of \lemref{MultiplicityInO2}. Assume $\beta\vdash k$ is some Young
diagram with $k$ boxes ($k\geq2$). We can find the irreducible constituents
of $\Res_{S_{k-2}\times S_{2}}^{S_{k}}S^{\beta}$ using the Littlewood-Richardson
rule. It is clear that every such constituent is of the form $S^{\gamma}\otimes\tr_{2}$
or $S^{\gamma}\otimes\sgn_{2}$ where $\gamma\vdash(k-2)$. Now inflating
this into a $S_{k-2}\times D_{4}$ representation, it is clear that
we get $S^{\gamma}\otimes\overline{\tr_{2}}$ or $S^{\gamma}\otimes\overline{\sgn_{2}}$
respectively. Now we need to induct a representation of this form
from $S_{k-2}\times D_{4}$ to $S_{k+2}$. However, by the transitivity
of induction, we know that 
\[
\Ind_{S_{k-2}\times D_{4}}^{S_{k+2}}W=\Ind_{S_{k-2}\times S_{4}}^{S_{k+2}}\Ind_{S_{k-2}\times D_{4}}^{S_{k-2}\times S_{4}}W.
\]
So we can at the first step induct to $S_{k-2}\times S_{4}$ and get
that 
\begin{align*}
\Ind_{S_{k-2}\times D_{4}}^{S_{k-2}\times S_{4}}(S^{\gamma}\otimes\overline{\tr_{2}}) & =S^{\gamma}\otimes(S^{[4]}\oplus S^{[2,2]})\\
\Ind_{S_{k-2}\times D_{4}}^{S_{k-2}\times S_{4}}(S^{\gamma}\otimes\overline{\sgn_{2}}) & =S^{\gamma}\otimes S^{[3,1]}.
\end{align*}
Finally we can induce these $S_{k-2}\times S_{4}$-representations
into a $S_{k+2}$ representation using the Littlewood-Richardson rule.
The explicit description of the Littlewood-Richardson rule in the
above procedure might be non-trivial, so we cannot say that we have
an explicit way to describe the multiplicity of $S^{\alpha}$ in $\Ind_{S_{k-2}\times D_{4}}^{S_{k+2}}\overline{\Res_{S_{k-2}\times S_{2}}^{S_{k}}S^{\beta}}$.
However, in certain cases what we obtained is enough as we are going
to see in the next section.
\end{rem}

\section{\label{sec:TheGlobalDimension}The global dimension}

In this section we will finally prove that the global dimension of
$\SE_{n}$ is $n-1$. According to the description of the quiver given
in \thmref{DescriptionOfTheQuiverOfPT_n}, it is clear that the longest
path in the quiver is of length $n-1$. Therefore it is clear that
\[
\gd\SE_{n}\leq n-1.
\]
For the opposite inequality it is enough to find one $\SE_{n}$-module
$M$ with
\[
\pd(M)=n-1.
\]
In this section we will prove that the $\SE_{n}$-module corresponding
to the Young diagram $[2,1^{n-2}]$ has projective dimension $n-1$.
We start with some notation. For $k\geq2$ we will denote the Young
diagram $[2,1^{k-2}]$ by $\ds_{k}$ (the ``$\ds$'' stands for
``dual standard'' since this module is just the tensor of the standard
representation with the sign representation). In the previous section
we denoted the sign representation of $S_{2}$ by $\sgn_{2}$. In
this section it will be convenient to denote the Young diagram $[1^{k}]$
by $\sgn_{k}$ for $k\geq1$. The major step will be to list all the
Jordan-Hölder factors of $P(\ds_{k})$.
\begin{lem}
\label{lem:FactorsAbove}Let $n\geq k$ and $\alpha\vdash r$ for
$r\geq k$. The module $S(\alpha)$ appears as a Jordan-Hölder factor
of $P(\ds_{k})$ if and only if $r=k$ and $\alpha=\ds_{k}$.
\end{lem}
\begin{proof}
Clear from the fact the the Cartan matrix is block unitriangular (\lemref{CartanMatrixIsUnitriangular}).
\end{proof}
\begin{lem}
\label{lem:JHFactorsOneBelow}Let $n\geq k\geq3$ and $\alpha\vdash(k-1)$.
The module $S(\alpha)$ appear as a Jordan-Hölder factor of $P(\ds_{k})$
if and only if $\alpha=\ds_{k-1}$ or $\alpha=\sgn_{k-1}$, each of
them has multiplicity $1$.
\end{lem}
\begin{proof}
Clear by \propref{CartanMatrixLength1}. The only way to obtain $\ds_{k}$
by adding two boxes but not in the same column is from $\sgn_{k-2}$
and $\sgn_{k-2}$ can be obtained by removing one box from $\ds_{k-1}$
or $\sgn_{k-1}$.
\end{proof}
\begin{lem}
\label{lem:NoElementsLength2FromDS}Let $n\geq k\geq3$. The module
$P(\ds_{k})$ has no Jordan-Hölder factors of the form $S(\alpha)$
with $\alpha\vdash(k-2)$.
\end{lem}
\begin{proof}
First assume $k=3$. From the quiver description (and \exaref{QuiverOfPT4})
it is clear that the only possible candidate is $\alpha=[1],$ the
trivial representation of $S_{1}$. But from \remref{RemarkOnLevel3}
we know that it is not a Jordan-Hölder factor of $P(\ds_{k})=P([2,1])$.
Now assume $k\geq4$. By \lemref{CombinedO1O2} we need to show that
$S^{\ds_{k}}$ is not an irreducible constituent in 
\[
\Ind_{S_{k-3}\times S_{3}}^{S_{k}}(\Res_{S_{k-3}}^{S_{k-2}}(S^{\alpha})\otimes\tr_{3})\oplus\Ind_{S_{k-4}\times D_{4}}^{S_{k}}\overline{\Res_{S_{k-4}\times S_{2}}^{S_{k-2}}S^{\alpha}}.
\]

First consider the module
\[
M_{1}=\Ind_{S_{k-3}\times S_{3}}^{S_{k}}(\Res_{S_{k-3}}^{S_{k-2}}(S^{\alpha})\otimes\tr_{3}).
\]
By Pieri's rule, a necessary condition for $S^{\ds_{k}}$ to be an
irreducible constituent in $M_{1}$ is that $\ds_{k}$ should be obtained
from some other diagram by adding three boxes, no two of them in the
same column (see \corref{MultiplicityInO1}). This is clearly impossible
so $S^{\ds_{k}}$ is not an irreducible constituent in $M_{1}$. Now
consider
\[
M_{2}=\Ind_{S_{k-4}\times D_{4}}^{S_{k}}\overline{\Res_{S_{k-4}\times S_{2}}^{S_{k-2}}S^{\alpha}}.
\]
Clearly, the $S_{k-4}\times S_{2}$-representation 
\[
\Res_{S_{k-4}\times S_{2}}^{S_{k-2}}S^{\alpha}
\]
consists of a direct sum of representations of the form $S^{\beta}\otimes S^{\gamma}$
where $\beta\vdash(k-4)$ and $\gamma\in\{[2],[1,1]\}$.

Now, by \remref{ComputationOfMultInO2} we know that in each $S_{k-2}\times S_{4}$-representation
\[
\Ind_{S_{k-4}\times D_{4}}^{S_{k-4}\times S_{4}}\overline{\Res_{S_{k-4}\times S_{2}}^{S_{k-2}}S^{\alpha}}
\]
will be a direct sum of representations of the form $S^{\beta}\otimes S^{\delta}$
where $\beta\vdash(k-4)$ and $\delta\in\{[4],[2,2],[3,1]\}$. It
is left to show that $\ds_{k}$ does not appear as an irreducible
constituent of 
\[
\Ind_{S_{k-4}\times S_{4}}^{S_{k}}(S^{\beta}\otimes S^{\delta}).
\]
In order to obtain the Young diagram $\ds_{k}$ from $\beta$, we
will have to add at least $3$ boxes in the same column so the skew
diagram $\ds_{k}/\beta$ has a column of length at least $3$. This
means that the content tableau should have $3$ rows but $\delta$
has at most $2$ rows. So by the Littlewood-Richardson rule, the $S_{k}$-representation
$S^{\ds_{k}}$ does not appear as an irreducible constituent in $\Ind_{S_{k-4}\times S_{4}}^{S_{k}}(S^{\beta}\otimes S^{\delta})$.
This finishes the proof. 
\end{proof}
Denote by $Q_{n}$ the quiver of $\SE_{n}$. Let $I$ be an admissible
ideal such that $A=\mathbb{C}(Q_{n})^{\ast}/I$ is Morita equivalent
to $\mathbb{C}\SE_{n}$. Note that we do not know much about this
ideal. Recall also that the number of times that a simple module $S(\beta)$
appears as a Jordan-Hölder factor in $P(\alpha)$ is the dimension
of the quotient space of all linear combinations of paths from $\alpha$
to $\beta$ modulo $I$.
\begin{lem}
\label{lem:NoLength2ImpliesNoLengthGraterThan2}Assume $k\geq3$.
Let $P(\alpha)$ be some projective module of $\mathbb{C}\SE_{n}$
for $\alpha\vdash k$. Assume that for every $\beta\vdash(k-2)$ the
simple module $S(\beta)$ is not a Jordan-Hölder factor of $P(\alpha)$.
Then for every $\beta\vdash r$ where $r<k-2$ the simple module $S(\beta)$
is not a Jordan-Hölder factor of $P(\alpha)$.
\end{lem}
\begin{proof}
Note that all the arrows in the quiver $Q_{n}$ are ``one step down'',
in other words, they are from diagrams with $m+1$ boxes to diagrams
with $m$ boxes. Now, if the simple module $S(\beta)$ is not a Jordan-Hölder
factor of $P(\alpha)$ for $\beta\vdash(k-2)$, this means that, modulo
the admissible ideal $I$, there are no non-trivial paths of length
$2$ starting from $\alpha$. In other words, all the paths in $(Q_{n})^{\ast}$
of length $2$ that start at $\alpha$ are elements of $I$. Therefore,
every path in $(Q_{n})^{\ast}$ of length greater than $2$ that start
at $\alpha$ is an element of $I$. Hence $S(\beta)$ is not a Jordan-Hölder
factor of $P(\alpha)$ if $\beta\vdash r$ where $r<k-2$.
\end{proof}
By \lemref{NoElementsLength2FromDS} and \lemref{NoLength2ImpliesNoLengthGraterThan2}
we get the following immediate corollary.
\begin{cor}
\label{cor:NoFactorsOfLength2AndAbove}Assume $k\geq3$. If $\alpha\vdash r$
for $r\leq k-2$ then $S(\alpha)$ is not a Jordan-Hölder factor of
$P(\ds_{k})$.
\end{cor}
Therefore, by \lemref{FactorsAbove}, \lemref{JHFactorsOneBelow}
and \corref{NoFactorsOfLength2AndAbove} we obtain the following result.
\begin{prop}
\label{prop:JHFactorsOfDulaStandard}Let $k\geq3$. The only Jordan-Hölder
factors of $P(\ds_{k})$ are $S(\ds_{k})$, $S(\ds_{k-1})$ and $S(\sgn_{k-1})$,
each has multiplicity $1$.
\end{prop}
We will need another observation.
\begin{lem}
\label{lem:SgnIsProjective}The simple module $S(\sgn_{k})$ equals
the projective module $P(\sgn_{k})$.
\end{lem}
\begin{proof}
Consider the description of $Q_{n}$ given in \thmref{DescriptionOfTheQuiverOfPT_n}.
It is easy to observe that one cannot add two boxes not in the same
column and obtain $\sgn_{k}$, so there are no arrows in $Q_{n}$
starting at $\sgn_{k}$. Therefore, the only path in $(Q_{n})^{\ast}$
which starts at $\sgn_{k}$ is the trivial one. This implies that
the only simple module appears as a Jordan-Hölder factor in $P(\sgn_{k})$
is $S(\sgn_{k})$ and it appears only once by \lemref{CartanMatrixIsUnitriangular}
so 
\[
P(\sgn_{k})=S(\sgn_{k})
\]
as required.
\end{proof}
Now we can prove our desired result using a classical argument of
homological dimension shift.
\begin{prop}
Assume $n\geq k\geq2$. Then
\[
\Ext^{k-1}(S(\ds_{k}),S([1]))\simeq\mathbb{C}
\]
where $S(\ds_{k})$ and $S([1])$ are $\SE_{n}$ representations.
\end{prop}
\begin{proof}
We will prove this by induction on $k$. For the base step $k=2$,
we know by the quiver description (\thmref{DescriptionOfTheQuiverOfPT_n})
that
\[
\Ext^{1}(S(\ds_{2}),S([1]))=\Ext^{1}(S([2]),S([1]))\simeq\mathbb{C}
\]

Now assume that 
\[
\Ext^{k-1}(S(\ds_{k}),S([1]))\simeq\mathbb{C}
\]
and we will prove that 
\[
\Ext^{k}(S(\ds_{k+1}),S([1]))\simeq\mathbb{C}.
\]
Consider the short exact sequence
\[
0\to K\to P(\ds_{k+1})\to S(\ds_{k+1})\to0.
\]
Where $K$ is some module (actually $K\simeq\Rad P(\ds_{k+1})$ but
we don't need this fact). By \propref{JHFactorsOfDulaStandard} we
know that the Jordan-Hölder factors of $K$ are $S(\ds_{k})$ and
$S(\sgn_{k})$. Note that there are no arrows in $Q_{n}$ from $\ds_{k}$
to $\sgn_{k}$ or vice versa (there are no arrows between any distinct
Young diagrams with the same number of boxes). Therefore,
\[
\Ext^{1}(S(\ds_{k}),S(\sgn_{k}))=\Ext^{1}(S(\sgn_{k}),S(\ds_{k}))=0
\]
and this implies that the only extension of these two modules is the
direct sum, so 
\[
K\simeq S(\ds_{k})\oplus S(\sgn_{k}).
\]
Therefore, the above short exact sequence is actually
\[
0\to S(\ds_{k})\oplus S(\sgn_{k})\to P(\ds_{k+1})\to S(\ds_{k+1})\to0.
\]

Now we use the long exact sequence theorem with the $\SE_{n}$-module
$S([1])$ we obtain the following exact sequence: 

\begin{align*}
0 & \to\Hom(S(\ds_{k+1}),S([1]))\to\Hom(P(\ds_{k+1}),S([1]))\to\Hom(S(\ds_{k})\oplus S(\sgn_{k}),S([1]))\to\\
 & \to\Ext^{1}(S(\ds_{k+1}),S([1]))\to\Ext^{1}(P(\ds_{k+1}),S([1]))\to\Ext^{1}(S(\ds_{k})\oplus S(\sgn_{k}),S([1]))\to\ldots\\
 & \to\Ext^{m}(S(\ds_{k+1}),S([1]))\to\Ext^{m}(P(\ds_{k+1}),S([1]))\to\Ext^{m}(S(\ds_{k})\oplus S(\sgn_{k}),S([1]))\to\ldots.
\end{align*}
Clearly, 
\[
\Ext^{m}(P(\ds_{k+1}),S([1]))=0
\]
for every $m\geq1$ since $P(\ds_{k+1})$ is projective. So obtain
the following exact sequence
\[
0\to\Ext^{k-1}(S(\ds_{k})\oplus S(\sgn_{k}),S([1]))\to\Ext^{k}(S(\ds_{k+1}),S([1]))\to0.
\]
Now, since $\Ext^{k-1}$ is an additive functor this equals
\[
0\to\Ext^{k-1}(S(\ds_{k}),S([1]))\oplus\Ext^{k-1}(S(\sgn_{k}),S([1]))\to\Ext^{k}(S(\ds_{k+1}),S([1]))\to0.
\]
However, by \lemref{SgnIsProjective}, we know that $S(\sgn_{k})$
is also projective so 
\[
\Ext^{k-1}(S(\sgn_{k}),S([1]))=0.
\]
Now, we remain with the exact sequence
\[
0\to\Ext^{k-1}(S(\ds_{k}),S([1]))\to\Ext^{k}(S(\ds_{k+1}),S([1]))\to0
\]
which implies that 
\[
\Ext^{k}(S(\ds_{k+1}),S([1]))\simeq\Ext^{k-1}(S(\ds_{k}),S([1]))\simeq\mathbb{C}
\]
as required.
\end{proof}
Since the longest path in $Q_{n}$ which start at $\ds_{k}$ is of
length $k-1$ we obtain immediately the following corollary.
\begin{cor}
Assume $n\geq k\geq2$. Then
\[
\pd\left(S(\ds_{k})\right)=k-1.
\]
\end{cor}
In particular, we obtain from here that 
\[
\pd(S(\ds_{n}))=n-1
\]
and since the longest path in $Q_{n}$ is of length $n-1$ we obtain
as an immediate corollary the goal of this paper.
\begin{cor}
The global dimension of $\mathbb{C}\SE_{n}$ and hence of $\mathbb{C}\E_{n}$
and $\mathbb{C}\PT_{n}$ is $n-1$ for $n\geq1$.
\end{cor}

\bibliographystyle{plain}
\bibliography{library}

\end{document}